\documentclass[11pt,oneside,onecolumn]{amsart}
\usepackage{amssymb}
\usepackage{xcolor}
\usepackage{mathrsfs}
\usepackage{amsfonts}
\usepackage{amsmath}
\usepackage{tikz}
\numberwithin{equation}{section}
\allowdisplaybreaks[4]
\newtheorem{Theorem}{Theorem}[section]

\newtheorem{Proposition}[Theorem]{Proposition}
\newtheorem{Corollary}[Theorem]{Corollary}
\newtheorem{Lemma}[Theorem]{Lemma}

\newtheorem{Definition}[Theorem]{Definition}

\unless\ifcsname proof\endcsname

\fi

\def\supp{{\rm supp}\ }

\def\bbZ{\mathbb{Z}}

\def\bbR{\mathbb{R}}

\begin{document}

\title{Weighted Estimates for Multilinear Fourier Multipliers}
\thanks{This work was supported partially by the
National Natural Science Foundation of China(10971105 and 10990012).}

\author{Kangwei Li}
\email{likangwei9@mail.nankai.edu.cn}

\address{School of Mathematical Sciences and LPMC,  Nankai University,
      Tianjin~300071, China}

\author{Wenchang Sun}

\email{sunwch@nankai.edu.cn}

\address{School of Mathematical Sciences and LPMC,  Nankai University,
      Tianjin~300071, China}

\begin{abstract}
We prove a H\"{o}rmander  type multiplier theorem for
multilinear Fourier multipiers with multiple weights.
We also
give   weighted estimates for their commutators with vector $BMO$ functions.
\end{abstract}

\maketitle

\section{Introduction and Main Results}
Multilinear Calder\'on-Zygmund theory was first studied in Coifman and Meyer's works \cite{CM,CM1}.
 Since Lacey and Thiele¡¯s work on the bilinear Hilbert transform \cite{LT1,LT2}, it
has been widely studied by many authors in harmonic analysis.
For an overview, we refer to \cite{DLY,DGY,DGGLY,GLY,GT,HY1,HY2,MW,YYZ} and references therein.

In this paper, we study the boundedness of multilinear Fourier multipliers.
Speficically, we consider the $N$-linear Fourier multiplier operator $T_m$ defined by
\begin{eqnarray*}
&&T_m(f_1,\cdots,f_N)(x)\\
&&\quad=\int_{\bbR^{Nn}}e^{2\pi ix\cdot(\xi_1+\cdots+\xi_N)}m(\xi_1,\cdots,\xi_N)\widehat{f_1}(\xi_1)\cdots \widehat{f_N}(\xi_N)d\xi_1\cdots d\xi_N
\end{eqnarray*}
for $f_1,\cdots,f_N\in \mathscr{S}(\bbR^n)$, where
$m\in C^s(\bbR^{Nn}\setminus \{0\})$ satisfies
\begin{equation}\label{eq:e1}
|\partial^{\alpha_1}_{\xi_1}\cdots \partial^{\alpha_N}_{\xi_N}m(\xi_1,\cdots,\xi_N)|
\le
C_{\alpha_1,\cdots,\alpha_N}(|\xi_1|+\cdots+|\xi_N|)^{-(|\alpha_1|+\cdots+|\alpha_N|)}
\end{equation}
for all $|\alpha_1|+\cdots+|\alpha_N|\le s$, and $N\ge 2$ is an integer.

In \cite{CM}, Coifman and Meyer proved that if $s$ is a sufficient large integer, then
$T_m$ is bounded from $L^{p_1}(\bbR^n)\times \cdots \times L^{p_N}(\bbR^n)$ to $L^p(\bbR^n)$ for
all $1<p_1,\cdots,p_N,p<\infty$ satisfying $1/{p_1}+\cdots+1/{p_N}=1/p$.

In \cite{T}, Tomita gave a H\"{o}rmander type theorem for multilinear Fourier multipliers.
As a consequence, $T_m$ is bounded from $L^{p_1}(\bbR^n)\times \cdots \times L^{p_N}(\bbR^n)$ to $L^p(\bbR^n)$ for
all $1<p_1,\cdots,p_N,p<\infty$ satisfying $1/{p_1}+\cdots+1/{p_N}=1/p$ with
$s=\lfloor Nn/2\rfloor+1$ in $(\ref{eq:e1})$, where $\lfloor Nn/2\rfloor$ is the integer part of $Nn/2$.
Grafakos and Si \cite{GS} gave similar results for the case $p\le 1$ by using   $L^r$-based Sobolev spaces,
$1 < r \le2$.

In \cite{FT}, Fujita and Tomita studied the weighted estimates of $T_m$ under the H\"{o}rmander condition and
classical $A_p$ weights.

In \cite{LOPTG}, Lerner,  Ombrosi, P\'{e}rez,  Torres and  Trujillo-Gonz\'{a}lez
introduced the $A_{\vec{P}}$ condition for multiple weights.
\begin{Definition}
Let $\vec{P}=(p_1,\cdots, p_N)$ with $1\le p_1, \cdots, p_N<\infty$ and $1/{p_1}+\cdots+1/{p_N}=1/p$. Given $\vec{w}=(w_1, \cdots, w_N)$. Set
\[
  v^{}_{\vec{w}}=\prod_{i=1}^N w_i^{p/{p_i}}.
\]
We say that $\vec{w}$ satisfies the $A_{\vec{P}}$ condition if
\begin{equation}\label{eq:e2}
\sup_Q\bigg( \frac{1}{|Q|}\int_Q v^{}_{\vec{w}}\bigg)^{1/p}
    \prod_{i=1}^N\bigg(\frac{1}{|Q|}\int_Q w_i^{1-p'_i}\bigg)^{1/{p_i'}}<\infty.
\end{equation}
When $p_i=1$, then $\bigg(\frac{1}{|Q|}\int_Q w_i^{1-p'_i}\bigg)^{1/{p_i'}}$ is understood as $(\inf_Q w_i)^{-1}$.
\end{Definition}

In \cite{BD}, Bui and Duong studied the boundedness of  $T_m$ with multiple weights under the condition (\ref{eq:e1}).
They also gave a result on commutators.

In this paper, we consider the weighted estimates of $T_m$ with multiple weights. Instead of (\ref{eq:e1}),
we  consider the H\"{o}rmander condition. Moreover, we do not assume that $s$ is an integer. To be precise, we prove the following.
\begin{Theorem}\label{thm:main}
Let $\vec{P}=(p_1,\cdots, p_N)$ with $1< p_1, \cdots, p_N<\infty$ and $1/{p_1}+\cdots+1/{p_N}=1/p$.  Suppose that
$Nn/2< s\le Nn$, that $m\in L^{\infty}(\bbR^{Nn})$ with
\begin{equation}\label{eq:e3}
\sup_{R>0}\|m(R\xi)\chi^{}_{\{1<|\xi|<2\}}\|_{H^s(\bbR^{Nn})}<\infty,
\end{equation}
that $r_0:=Nn/s<p_1,\cdots,p_N<\infty$ and that $\vec{w}\in A_{\vec{P}/{r_0}}$. Then
\begin{eqnarray}
&& \|T_m(\vec{f})\|_{L^p(v^{}_{\vec{w}})}\le C \prod_{i=1}^N \|f_i\|_{L^{p_i}(w_i)}.\label{eq:e4}
\end{eqnarray}
\end{Theorem}
Recall that the Sobolev space $H^s$ consists of all $f\in\mathscr{S}'$ such that
\[
  \|f\|_{H^s}:=\|(I-\Delta)^{s/2}f\|_{L^2}<\infty,
\]
where $(I-\Delta)^{s/2}f=\mathcal{F}^{-1}((1+4\pi^2 |\xi|^2)^{s/2}\hat{f}(\xi))$. If $s$ is an integer, then
$\|f\|_{H^s}\asymp \sum_{|\alpha|\le s}\|\partial^\alpha f\|_{L^2}$.

\begin{Corollary}\label{cor:c1}
Let $1<p_1, \cdots, p_N<\infty$ with $1/{p_1}+\cdots+1/{p_N}=1/p$. Suppose that
$Nn/2< s\le Nn$, that $m\in L^\infty(\bbR^{Nn})$ satisfies $(\ref{eq:e3})$, that $1<p<\infty$, that $p^{}_{\beta}:=\min\{p_1,\cdots,p_N\}<r_0'=(Nn/s)'$ and that
\[
  (w_1,\cdots,w^{}_{\beta-1}, v_{\vec{w}}^{1-p'},w^{}_{\beta+1},\cdots,w_N)\in A_{\vec{\tilde{P}}/{r_0}},
\]
where $\vec{\tilde{P}}=(p_1,\cdots,p^{}_{\beta-1},p',p^{}_{\beta+1},\cdots,p_N)$. Then
$(\ref{eq:e4})$ holds.
\end{Corollary}

Commutators are a class of non-convolution operators \cite{CM1,DGY,G,HMY}.
Here we consider the commutator of a vector $BMO$ function and the multilinear operator with multiple weights. Given a locally
integrable vector function $\vec{b}=(b_1,\cdots, b_N)$, we define the $N$-linear commutator of $\vec{b}$ and
$N$-linear operator $T_m$ by
\[
  T_{m;\vec{b}}(\vec{f})=\sum_{i=1}^NT_{m;\vec{b}}^i(\vec{f}),
\]
where
\[
  T_{m;\vec{b}}^i(\vec{f})=b_iT_m(\vec{f})-T_m(f_1,\cdots,b_if_i,\cdots,f_N).
\]
If $\vec{b}\in BMO^N$, define $\|\vec{b}\|_{BMO^N}=\sup_{i=1,\cdots,N}\|b_i\|_{BMO}$.

\begin{Theorem}\label{thm:main1}
Under the hypotheses of Theorem~\ref{thm:main}. If moreover $\vec{b}\in BMO^N$, then
\begin{equation}\label{eq:commu}
\|T_{m;\vec{b}}(\vec{f})\|_{L^p(v^{}_{\vec{w}})}\le C \|\vec{b}\|_{BMO^N}\prod_{i=1}^N \|f_i\|_{L^{p_i}(w_i)}.
\end{equation}
\end{Theorem}

\begin{Corollary}\label{cor:c2}
Under the hypotheses of Corollary~\ref{cor:c1}. If moreover  $\vec{b}\in BMO^N$, then $(\ref{eq:commu})$ holds.
\end{Corollary}

In the rest of this paper, we give proofs for the above results.
We write $A\lesssim B$ if $A\le CB$ for some positive constant $C$, depending on $N$, the dimension $n$, the Lebesgue exponents and possibly the weights. We
write $A\asymp B$ if $A\lesssim B$ and $B\lesssim A$.

\section{Proof of the Main Results}
We begin with the definition of the Hardy-Littlewood maximal function,
\[
  M(f)(x)=\sup_{Q\ni x}\frac{1}{|Q|}\int_Q |f(y)|dy.
\]
The sharp maximal function is defined by
\[
  M^\sharp(f)(x)=\sup_{Q\ni x}\inf_{c\in\bbR}\frac{1}{|Q|}\int_Q |f(y)-c|dy.
\]
For $\delta>0$, we also need the maximal functions 
\[
  M_\delta(f)=M(|f|^\delta)^{1/\delta}\quad\mbox{and}\quad M_\delta^\sharp (f)=M^\sharp(|f|^\delta)^{1/\delta}.
\]

We use the following form of a classical result by Fefferman and Stein \cite{FS}.
\begin{Proposition}\label{prop:p1}
Let $0<p,\delta<\infty$ and $w\in A_{\infty}$. Then there exists some constant $C_{n,p,\delta,w}$ such that
\[
  \int_{\bbR^n}(M_\delta f)(x)^pw(x)dx\le C_{n,p,\delta,w}\int_{\bbR^n}(M_\delta^\sharp f)(x)^pw(x)dx.
\]
\end{Proposition}

For $\vec{f}=(f_1,\cdots,f_N)$ and $p\ge 1$, we define
\[
  \mathcal{M}_p(\vec{f})=\sup_{Q\ni x} \prod_{i=1}^N\bigg(\frac{1}{|Q|}\int_Q |f_i(y_i)|^p dy_i\bigg)^{1/p}.
\]
The following proposition gives a necessary and sufficient condition for the boundedness
of $\mathcal{M}_p$.
\begin{Proposition}\cite[Proposition 2.3]{BD}\label{prop:bd}
Let $p_0\ge 1$ and $p_i>p_0$ for all $i=1,\cdots, N$ and $1/{p_1}+\cdots+1/{p_N}=1/p$. Then the inequality
\[
   \|\mathcal{M}_{p_0}(\vec{f})\|_{L^p(v^{}_{\vec{w}})}\le C \prod_{i=1}^N \|f_i\|_{L^{p_i}(w_i)}
\]
holds if and only if $\vec{w}\in A_{\vec{P}/{p_0}}$, where $\vec{P}/{p_0}=(p_1/{p_0},\cdots, p_N/{p_0})$.
\end{Proposition}

Next we introduce two properties of   multiple weights.
\begin{Proposition}\cite[Theorem 3.6]{LOPTG}
Let $\vec{w}=(w_1,\cdots,w_N)$ and $1\le p_1$, $\cdots$, $p_N<\infty$. Then $\vec{w}\in A_{\vec{P}}$ if and only if
\begin{equation}
\begin{cases}
w_i^{1-p_i'}\in A_{Np_i'}, & i=1,\cdots N, \\
v_{\vec{w}}\in A_{Np},
\end{cases}
\end{equation}
where the condition $w_i^{1-p_i'}\in A_{Np_i'}$ in the case $p_i=1$ is understood as $w_i^{1/N}\in A_1$.
\end{Proposition}

The following result appears in \cite[Lemma 6.1]{LOPTG}. For our purpose,
we make a slight change.
\begin{Proposition}\label{prop:loptg}
Assume that $\vec{w}=(w_1,\cdots,w_N)$ satisfies the $A_{\vec{P}}$ condition, where $\vec{P}=(p_1,\cdots, p_N)$ with
$1<p_1,\cdots, p_N<\infty$. Let
$Nn/2<s\le Nn$. Then there exists a  constant
$1<r<\min\{p_1$, $\cdots$, $p_N$, $s/(s-1)$, $2s/(Nn)\}$ such that $\vec{w}\in A_{\vec{P}/r}$.
\end{Proposition}

\begin{proof}
Since we need an accurate estimate of $r$, we sketch the proof given in  \cite{LOPTG}.
Using the reverse H\"{o}lder inequality, it was shown in  the proof of  \cite[Lemma 6.1]{LOPTG} that there exist constants $c_i, t_i>1$ such that
\[
   \bigg(\frac{1}{|Q|}\int_Q w_i^{-\frac{t_i}{p_i-1}}\bigg)^{1/{t_i}}\le
   \frac{c_i}{|Q|}\int_Q w_i^{-\frac{1}{p_i-1}}
\]
for all $i=1,\cdots, N$.
Let $r_i$ be selected such that
\[
  \frac{t_i}{p_i-1}=\frac{1}{\frac{p_i}{r_i}-1}.
\]
Then $r=\min\{r_1,\cdots, r_N\}$ satisfies $\vec{w}\in A_{\vec{P}/r}$.
By H\"{o}lder's inequality, we can choose $t_i$, and therefore $r$, arbitrarily close to $1$. Since both $s/(s-1)$ and $2s/(Nn)$ are greater than $1$,
we get the desired conclusion.
\end{proof}

The boundedness of multilinear Fourier multipliers
was proved in \cite{CM,GS,GT,T}. Here we cite a version in \cite{GS}.

\begin{Proposition}\cite[Theorem 1.1]{GS}\label{prop:gs}
Suppose that $m\in L^\infty(\bbR^{Nn})$ satisfies $(\ref{eq:e3})$.
Let $Nn/2<s\le Nn$, $Nn/s< p_1, \cdots, p_N<\infty$ and $1/{p_1}+\cdots+1/{p_N}=1/p$.  Then
$T_m$ is bounded from $L^{p_1}(\bbR^n)\times \cdots \times L^{p_N}(\bbR^n)$ to $L^p(\bbR^n)$.
\end{Proposition}

The following lemma is the key to our main results.

\begin{Lemma}\label{lm:main}
Under the hypotheses of Theorem~\ref{thm:main}, if moreover $0<\delta<p_0/N$, where $p_0=rr_0$ and $r$ is the same as that appears in Proposition~\ref{prop:loptg}. Then for all $\vec{f}$ in product
of $L^{q_i}(\bbR^n)$ spaces with $p_0\le q_1,\cdots, q_N<\infty$,
\begin{equation}\label{eq:e5}
M_\delta^\sharp(T_m(\vec{f}))\le C\mathcal{M}_{p_0}(\vec{f}).
\end{equation}
\end{Lemma}
\begin{proof}
By Proposition~\ref{prop:loptg}, $rr_0\le 2$. Consequently, $p_0/N = rr_0/N\le 1$.
Fix a point $x$ and a cube $Q$ such that $x\in Q$. It suffices to prove
\begin{equation}\label{eq:e6}
\bigg( \frac{1}{|Q|}\int_Q |T_m(\vec{f})(z)-c_Q|^\delta dz\bigg)^{1/\delta}\le C\mathcal{M}_{p_0}(\vec{f})(x)
\end{equation}
for some constant $c_Q$ to be determined later
since $||\alpha|^\delta-|\beta|^\delta|\le |\alpha-\beta|^\delta$ for $0<\delta<1$.
Following the method used in \cite{LOPTG}, let $f_i=f_i^0+f_i^\infty$, where $f_i^0=f_i\chi^{}_{Q^*}$ for all
$i=1,\cdots,N$, and $Q^*=4\sqrt{n}Q$. Then
\begin{eqnarray*}
\prod_{i=1}^Nf_i(y_i)&=&\prod_{i=1}^N(f_i^0(y_i)+f_i^\infty(y_i))\\
&=&\sum_{\alpha_1,\cdots,\alpha_N\in\{0,\infty\}}f_1^{\alpha_1}(y_1)\cdots f_N^{\alpha_N}(y_N)\\
&=&\prod_{i=1}^N f_i^0(y_i)+\sum_{\alpha_1,\cdots,\alpha_N\in \mathcal{I}} f_1^{\alpha_1}(y_1)\cdots f_N^{\alpha_N}(y_N),
\end{eqnarray*}
where $\mathcal{I}:=\{\alpha_1,\cdots,\alpha_N: \mbox{there is at least one $\alpha_i\neq 0$}\}$.
Write then
\begin{equation}\label{eq:e7}
T_m(\vec{f})(z)=T_m(\vec{f^0})(z)+\sum_{\alpha_1,\cdots,\alpha_N\in\mathcal{I}}
T_m(f_1^{\alpha_1},\cdots,f_N^{\alpha_N})(z)
\end{equation}
Applying  Kolmogorov's inequality to the first term
\[
  T_m(\vec{f^0})(z)=T_m(f_1^0,\cdots,f_N^0)(z),
\]
we have
\begin{eqnarray*}
\bigg(\frac{1}{|Q|}\int_Q |T_m(\vec{f^0})(z)|^\delta dz\bigg)^{1/\delta}
&\lesssim& \|T_m(\vec{f^0})(z)\|_{L^{p_0/N,\infty}(Q,dx/{|Q|})}\\
&\lesssim& \prod_{i=1}^N\bigg(\frac{1}{|Q^*|}\int_{Q^*}|f_i(y_i)|^{p_0}dy_i\bigg)^{1/{p_0}}\\
&\le& \mathcal{M}_{p_0}(\vec{f})(x),
\end{eqnarray*}
since $p_0>Nn/s$ and $T_m$ is bounded from $L^{p_0}\times \cdots\times L^{p_0}$ to $L^{p_0/N}$, thanks to Proposition~\ref{prop:gs}.

In order to study the other terms in $(\ref{eq:e7})$, we set now
\[
  c=\sum_{\alpha_1,\cdots,\alpha_N\in\mathcal{I}}
T_m(f_1^{\alpha_1},\cdots,f_N^{\alpha_N})(x),
\]
and we will show that, for any $z\in Q$, we also get an estimate of the form
\begin{eqnarray}
\sum_{\alpha_1,\cdots,\alpha_N\in\mathcal{I}}
|T_m(f_1^{\alpha_1},\cdots,f_N^{\alpha_N})(z)-T_m(f_1^{\alpha_1},\cdots,f_N^{\alpha_N})(x)|\le C\mathcal{M}_{p_0}(\vec{f})(x).   \label{eq:e8}
\end{eqnarray}
Consider first the case when $\alpha_1=\cdots=\alpha_N=\infty$ and define
\[
  T_m(\vec{f^\infty})(z)=T_m(f_1^\infty,\cdots,f_N^\infty)(z).
\]
Let $m_j=m(\cdot)\psi(\cdot/{2^j})$, where $\psi\in\mathscr{S}(\bbR^{Nn})$ with $\supp \psi\subset\{\xi\in\bbR^{Nn}: 1/2\le|\xi|\le2\}$ and
\[
  \sum_{j\in\bbZ}\psi(2^{-j}\xi)=1,\quad\mbox{$\xi\neq 0$.}
\]
We have
\begin{eqnarray*}
&&|T_m(\vec{f^\infty})(z)-T_m(\vec{f^\infty})(x)|\\
&\le&\sum_{j\in\bbZ}|T_{m_j}(\vec{f^\infty})(z)-T_{m_j}(\vec{f^\infty})(x)|\\
&\le&\sum_{j\in\bbZ}\int_{\bbR^{Nn}\setminus (Q^*)^N}|\breve{m_j}(z-y_1, \cdots, z-y_N)-\breve{m_j}(x-y_1, \cdots, x-y_N)|\\
&&\quad\cdot
\prod_{i=1}^N|f_i(y_i)|d\vec{y}\\
&=&\sum_{j\in\bbZ}\sum_{k=0}^\infty\int_{(2^{k+1}Q^*)^N\setminus (2^kQ^*)^N}|\breve{m_j}(z-y_1, \cdots, z-y_N)\\
&&\quad-\breve{m_j}(x-y_1, \cdots, x-y_N)|\cdot
\prod_{i=1}^N|f_i(y_i)|d\vec{y}\\
&:=&\sum_{k=0}^\infty I_k.
\end{eqnarray*}
For any $k=0,1,2,\cdots$, we have
\begin{eqnarray*}
I_k&=&\sum_{j\in\bbZ}\int_{(2^{k+1}Q^*)^N\setminus (2^kQ^*)^N}|\breve{m_j}(z-y_1, \cdots, z-y_N)\\
&&\quad-\breve{m_j}(x-y_1, \cdots, x-y_N)|\cdot
\prod_{i=1}^N|f_i(y_i)|d\vec{y}\\
&\le& \sum_{j\in\bbZ}\bigg(\int_{(2^{k+1}Q^*)^N\setminus (2^kQ^*)^N}|\breve{m_j}(z-y_1, \cdots, z-y_N)\\
&&\quad
\!-\breve{m_j}(x-y_1, \cdots, x-y_N)|^{p_0'}d\vec{y}\bigg)^{1/{p_0'}} \!\!
\bigg(\int_{(2^{k+1}Q^*)^N}\prod_{i=1}^N|f_i(y_i)|^{p_0}d\vec{y}\bigg)^{1/{p_0}}\\
&:=&\sum_{j\in\bbZ} J_{k,j}\cdot\bigg(\int_{(2^{k+1}Q^*)^N}\prod_{i=1}^N|f_i(y_i)|^{p_0}d\vec{y}\bigg)^{1/{p_0}}.
\end{eqnarray*}
Let $h=z-x$ and $\tilde{Q}=x-Q^*$. We have
\begin{eqnarray*}
J_{j,k}&=&\bigg(\int_{(2^{k+1}Q^*)^N\setminus (2^kQ^*)^N}|\breve{m_j}(z-y_1, \cdots, z-y_N)\\ &&\quad
-\breve{m_j}(x-y_1, \cdots, x-y_N)|^{p_0'}d\vec{y}\bigg)^{1/{p_0'}}\\
&=&\bigg(\int\limits_{(2^{k+1}\tilde{Q})^N\setminus (2^k\tilde{Q})^N}
  \hskip -14pt |\breve{m_j}(h+y_1, \cdots, h+y_N)
-\breve{m_j}(y_1, \cdots, y_N)|^{p_0'}d\vec{y}\bigg)^{1/{p_0'}}\\
&\le&2\bigg(\int_{c_1 2^kl(Q)\le  |y|\le c_2 2^k l(Q)}|\breve{m_j}(y_1, \cdots, y_N)|^{p_0'}d\vec{y}\bigg)^{1/{p_0'}}\\
&\lesssim&(2^{k} l(Q))^{-s}\bigg(\int_{c_1 2^kl(Q)\le |y|\le  c_2 2^k l(Q)}(4\pi^2|y_1|^2+\cdots+4\pi^2|y_N|^2)^{sp_0'/2}\\
&&\quad\cdot|\breve{m_j}(y_1, \cdots, y_N)|^{p_0'}d\vec{y}\bigg)^{1/{p_0'}}\\
&\le&(2^{k} l(Q))^{-s}\bigg(\int_{\bbR^{Nn}}(4\pi^2|y_1|^2+\cdots+4\pi^2|y_N|^2)^{sp_0'/2}\\
&&\quad\cdot|2^{-jNn}\breve{m_j}(2^{-j}y_1, \cdots, 2^{-j}y_N)|^{p_0'}d\vec{y}\bigg)^{1/{p_0'}}2^{j(Nn/{p_0}-s)}\\
&\le&(2^{k} l(Q))^{-s}\bigg(\int_{\bbR^{Nn}}(1+4\pi^2|y_1|^2+\cdots+4\pi^2|y_N|^2)^{sp_0'/2}\\
&&\quad\cdot|2^{-jNn}\breve{m_j}(2^{-j}y_1, \cdots, 2^{-j}y_N)|^{p_0'}d\vec{y}\bigg)^{1/{p_0'}}2^{j(Nn/{p_0}-s)}\\
&\lesssim&(2^k l(Q))^{-s}2^{j(Nn/{p_0}-s)}\|m(2^j\cdot)\psi\|_{H^s},
\end{eqnarray*}
where Schwarz's inequality, Young's inequality and Plancherel's theorem are used in the last step, see \cite[Lemma 3.3]{T}.
Suppose that $2^{-l}\le l(Q)<2^{-l+1}$. Then we have
\begin{eqnarray}
\sum_{j\ge l}J_{j,k}&\lesssim&\sup_{j}\|m(2^j\cdot)\psi\|_{H^s}\sum_{j\ge l}(2^k l(Q))^{-s}2^{j(Nn/{p_0}-s)}\label{eq:e9}\\
&\lesssim&\sup_{R>0}\|m(R\xi)\chi^{}_{\{1<|\xi|<2\}}\|_{H^s}2^{-ks}l(Q)^{-Nn/{p_0}}.\nonumber
\end{eqnarray}
On the other hand, we also have
\begin{eqnarray*}
J_{j,k}
&=&\bigg(\int\limits_{(2^{k+1}\tilde{Q})^N\setminus (2^k\tilde{Q})^N}
  \hskip -14pt |\breve{m_j}(h+y_1, \cdots, h+y_N)
-\breve{m_j}(y_1, \cdots, y_N)|^{p_0'}d\vec{y}\bigg)^{1/{p_0'}}\\
&\le&
  \bigg(\int_{(2^{k+1}\tilde{Q})^N\setminus (2^k\tilde{Q})^N}
   \!\!\bigg(\!\!\int_0^1\!|\vec{h}\cdot\nabla\breve{m_j}(y_1+\theta h, \cdots, y_N+\theta h)|d\theta\bigg)^{p_0'}d\vec{y}\bigg)^{1/{p_0'}}\\
&\le&\int_0^1\bigg(\int_{(2^{k+1}\tilde{Q})^N\setminus (2^k\tilde{Q})^N}|\vec{h}\cdot\nabla\breve{m_j}(y_1+\theta h, \cdots, y_N+\theta h)|^{p_0'}d\vec{y}\bigg)^{1/{p_0'}}d\theta\\
&\le&\bigg(\int_{c_1 2^kl(Q)\le |y|\le c_2 2^k l(Q)}|\vec{h}\cdot\nabla\breve{m_j}(y_1, \cdots, y_N)|^{p_0'}d\vec{y}\bigg)^{1/{p_0'}},
\end{eqnarray*}
where $\vec{h}=(h,\cdots,h)\in\bbR^{Nn}$. Since
\begin{eqnarray*}
\vec{h}\cdot\nabla\breve{m_j}(y_1, \cdots, y_N)&=&\sum_{r=1}^{Nn}h_r\partial_r\breve{m_j}(y_1, \cdots, y_N),
\end{eqnarray*}
we have
\begin{eqnarray*}
J_{j,k}&\lesssim&\sum_{r=1}^{Nn}l(Q)\bigg(\int_{c_1 2^kl(Q)\le |y|\le c_2 2^k l(Q)}|\partial_r\breve{m_j}(y_1, \cdots, y_N)|^{p_0'}d\vec{y}\bigg)^{1/{p_0'}}\\
&\lesssim&\sum_{r=1}^{Nn}l(Q)(2^{k} l(Q))^{-s}\bigg(\int_{\bbR^{Nn}}(1+4\pi^2|y_1|^2+\cdots+4\pi^2|y_N|^2)^{sp_0'/2}\\
&&\quad\cdot|2^{-jNn}\partial_r\breve{m_j}(2^{-j}y_1, \cdots, 2^{-j}y_N)|^{p_0'}d\vec{y}\bigg)^{1/{p_0'}}2^{j(Nn/{p_0}-s)}\\
&\lesssim&\sum_{r=1}^{Nn}l(Q)(2^{k} l(Q))^{-s}2^{j(Nn/{p_0}-s)}2^j\|m(2^j\xi)\xi_r\psi(\xi)\|_{H^s}\\
&\lesssim&\sup_{R>0}\|m(R\xi)\chi^{}_{\{1<|\xi|<2\}}\|_{H^s}l(Q)(2^{k} l(Q))^{-s}2^{j(Nn/{p_0}-s)}2^j.
\end{eqnarray*}
By Proposition~\ref{prop:loptg}, $Nn/{p_0}>s-1$. It follows that
\begin{equation}\label{eq:e10}
\sum_{j<l} J_{j,k}\lesssim\sup_{R>0}\|m(R\xi)\chi^{}_{\{1<|\xi|<2\}}\|_{H^s}2^{-ks}l(Q)^{-Nn/{p_0}}.
\end{equation}
Combining the arguments above we get
\begin{eqnarray*}
|T_m(\vec{f^\infty})(z)-T_m(\vec{f^\infty})(x)|
&\lesssim&\sum_{k=0}^{\infty}2^{-k(s-Nn/{p_0})}\mathcal{M}_{p_0}(\vec{f})\\
&\lesssim&\mathcal{M}_{p_0}(\vec{f}).
\end{eqnarray*}

What remains to be considered are the terms in (\ref{eq:e8}) such that $\alpha_{i_1}=\cdots=\alpha_{i_\gamma}=0$
for some $\{i_1,\cdots,i_\gamma\}\subset\{1,\cdots,N\}$ and $1\le \gamma<N$. We have
\begin{eqnarray*}
&&|T_m(f_1^{\alpha_1},\cdots,f_N^{\alpha_N})(z)-T_m(f_1^{\alpha_1},\cdots,f_N^{\alpha_N})(x)|\\
&\le&\sum_j|T_{m_j}(f_1^{\alpha_1},\cdots,f_N^{\alpha_N})(z)-T_{m_j}(f_1^{\alpha_1},\cdots,f_N^{\alpha_N})(x)|\\
&\le&\sum_j\prod_{i\in\{i_1,\cdots,i_\gamma\}}\int_{Q^*}|f_i(y_i)|dy_i\int_{(\bbR^{n}\setminus Q^*)^{N-\gamma}}|\breve{m_j}(z-y_1, \cdots, z-y_N)\\&&\quad-\breve{m_j}(x-y_1, \cdots, x-y_N)|
\prod_{i\notin \{i_1,\cdots,i_\gamma\}}|f_i(y_i)|dy_i\\
&=&\sum_j\sum_{k=0}^\infty\prod_{i\in\{i_1,\cdots,i_\gamma\}}\int_{Q^*}|f_i(y_i)|dy_i
 \int\limits_{(2^{k+1}Q^*\setminus 2^kQ^*)^{N-\gamma}}
 \hskip -16pt |\breve{m_j}(z-y_1, \cdots, z-y_N)\\&&\quad-\breve{m_j}(x-y_1, \cdots, x-y_N)|
\prod_{i\notin \{i_1,\cdots,i_\gamma\}}|f_i(y_i)|dy_i\\
&\le&\sum_j\sum_{k=0}^\infty\bigg(\int_{(Q^*)^\gamma\times(2^{k+1}Q^*\setminus 2^kQ^*)^{N-\gamma}}|\breve{m_j}(z-y_1, \cdots, z-y_N)\\&&\quad-\breve{m_j}(x-y_1, \cdots, x-y_N)|^{p_0'}d\vec{y}\bigg)^{1/{p_0'}}
\bigg(\int_{(2^{k+1}Q^*)^N}\prod_{i=1}^N|f_i(y_i)|^{p_0}d\vec{y}\bigg)^{1/{p_0}}.
\end{eqnarray*}
Then by similar arguments as above we get that
\[
  |T_m(f_1^{\alpha_1},\cdots,f_N^{\alpha_N})(z)-T_m(f_1^{\alpha_1},\cdots,f_N^{\alpha_N})(x)|\le C\mathcal{M}_{p_0}(\vec{f})(x).
\]
This completes the proof.
\end{proof}

Now we are ready to prove Theorem~\ref{thm:main}.
\begin{proof}[Proof of Theorem~\ref{thm:main}]
By  Proposition~\ref{prop:p1} and Lemma~\ref{lm:main}, we have
\begin{eqnarray*}
 \|T_m (\vec{f})\|_{L^p(v_{\vec{w}})}\le \|M_\delta(T_m (\vec{f}))\|_{L^p(v_{\vec{w}})}
 &\le& C_{n,p,\delta, \vec{w}}\|M_\delta^\sharp(T_m (\vec{f}))\|_{L^p(v_{\vec{w}})}\\
 &\le& C\|\mathcal{M}_{p_0}(\vec{f})\|_{L^p(v_{\vec{w}})}.
\end{eqnarray*}
  Now the desired conclusion follows from
Proposition~\ref{prop:bd}.
\end{proof}

\begin{proof}[Proof of Corollary~\ref{cor:c1}]
Without loss of generality  assume that
\[
  p_1=\min\{p_1,\cdots, p_N\}.
\] 
  As in \cite{FT}, we set
\[
  m_1=m(-(\xi_1+\cdots+\xi_N),\xi_2,\cdots,\xi_N).
\]
Then we can write
\begin{equation}\label{eq:e11}
\int_{\bbR^n}T_{m}(f_1,\cdots,f_N)gdx=\int_{\bbR^n} T_{m_1}(g,f_2,\cdots,f_N)f_1 dx
\end{equation}
for all $g$, $f_1,\cdots,f_N\in\mathscr{S}(\bbR^n)$.
By a change of variables we get (see also \ref{FT})
\[
  \sup_{R>0}\|m_1(R\xi)\chi^{}_{\{1<|\xi|<2\}}\|_{H^s(\bbR^{Nn})}\le
  C\sup_{R>0}\|m(R\xi)\chi^{}_{\{1<|\xi|<2\}}\|_{H^s(\bbR^{Nn})}<\infty.
\]
Since $1/{p_1}+\cdots+1/{p_N}=1/p$, we have $1/{p'}+1/{p_2}+\cdots+1/{p_N}=1/{p_1'}$. Therefore,
$Nn/s<p'_1<\min\{p', p_2, \cdots, p_N\}$ due to $\min\{p_1,\cdots, p_N\}<(Nn/s)'$. Since
\[
  v_{\vec{w}}^{(1-p')p_1'/{p'}}w_2^{p_1'/{p_2}}\cdots w_N^{p_1'/{p_N}}=w_1^{1-p_1'},
\]
it follows from Theorem~\ref{thm:main} that
\begin{equation}\label{eq:e12}
\|T_{m_1}(\vec{f})\|_{L^{p_1'}(w_1^{1-p_1'})}\le C \|f_1\|_{L^{p'}(v_{\vec{w}}^{1-p'})}\prod_{i=2}^N \|f_i\|_{L^{p_i}(w_i)}.
\end{equation}
Then by duality and $(\ref{eq:e11})$ and $(\ref{eq:e12})$, we have
\begin{eqnarray*}
\|T_m(\vec{f})\|_{L^p(v^{}_{\vec{w}})}&=&
         \sup_{\|g\|_{L^{p'}(v_{\vec{w}}^{1-p'})}=1}|\langle T_{m}(f_1,\cdots,f_N),g\rangle|\\
&=&
         \sup_{\|g\|_{L^{p'}(v_{\vec{w}}^{1-p'})}=1}|\langle T_{m_1}(g,\cdots,f_N),f_1\rangle|\\
&\le&
         \sup_{\|g\|_{L^{p'}(v_{\vec{w}}^{1-p'})}=1}\|T_{m_1}(g,f_2,\cdots,f_N)\|_{L^{p_1'}(w_1^{1-p_1'})}
         \|f_1\|_{L^{p_1}(w_1)}\\
&\le&    C \prod_{i=1}^N \|f_i\|_{L^{p_i}(w_i)}.
\end{eqnarray*}
This completes the proof.
\end{proof}

\begin{Lemma}\label{lm:main1}
Suppose that $m\in L^\infty(\bbR^{Nn})$ satisfies $(\ref{eq:e3})$, that $\vec{b}\in BMO^N$ and that $0<\delta<\epsilon<p_0/N$. Then for
any $q_0>p_0$, there exists some constant $C>0$ such that
\[
  M_\delta^\sharp (T_{m;\vec{b}}(\vec{f}))(x)\le C\|\vec{b}\|_{BMO^N}(M_\epsilon(T_m(\vec{f}))(x)+ \mathcal{M}_{q_0}(\vec{f})(x))
\]
for all $N$-tuples $\vec{f}=(f_1,\cdots,f_N)$ of bounded measurable functions with compact support.
\end{Lemma}
\begin{proof}
By linearity it suffices to consider the case of $\vec{b}=(b,0,\cdots,0)\in BMO^N$. Fix $b\in BMO$ and
consider the operator
\[
  T_{m;b}(\vec{f})=bT_m(\vec{f})-T_m(bf_1,f_2,\cdots,f_N).
\]
Fix $x\in\bbR^n$. For any cube $Q$ centered at $x$, set $Q^*=4\sqrt{n}Q$. Then we have
\[
   T_{m;b}(\vec{f})=(b-b_{Q^*})T_m(\vec{f})-T_m((b-b_{Q^*})f_1,f_2,\cdots,f_N).
\]
Since $0<\delta<1$, we have
\begin{eqnarray*}
&&\bigg(\frac{1}{|Q|}\int_Q \bigg||T_{m;b}(\vec{f})(z)|^\delta -|c|^\delta\bigg|dz\bigg)^{1/\delta}\\
&\le& \bigg(\frac{1}{|Q|}\int_Q |T_{m;b}(\vec{f})(z) -c|^\delta dz\bigg)^{1/\delta}\\
&\lesssim&\bigg(\frac{1}{|Q|}\int_Q |(b-b_{Q^*})T_m(\vec{f})(z)|^\delta dz\bigg)^{1/\delta}\\
&&\quad+\bigg(\frac{1}{|Q|}\int_Q |T_m((b-b_{Q^*})f_1,f_2,\cdots,f_N)(z) -c|^\delta dz\bigg)^{1/\delta}\\
&:=&I+II.
\end{eqnarray*}
For any $1<q<\epsilon/\delta$, by H\"{o}lder and John-Nirenberg inequality, we have
\begin{eqnarray*}
I&\le& \bigg(\frac{1}{|Q|}\int_Q |b-b_{Q^*}|^{q'\delta} dz\bigg)^{1/{q'\delta}}
       \bigg(\frac{1}{|Q|}\int_Q |T_m(\vec{f})(z)|^{q\delta} dz\bigg)^{1/{q\delta}}\\
 &\lesssim& \|b\|^{}_{BMO}M_\epsilon(T_m(\vec{f}))(x).
\end{eqnarray*}
Using the similar decomposition as that in the proof of Lemma~\ref{lm:main}, we can write
\begin{eqnarray*}
\prod_{i=1}^Nf_i(y_i)&=&\sum_{\alpha_1,\cdots,\alpha_N\in\{0,\infty\}}f_1^{\alpha_1}(y_1)\cdots f_N^{\alpha_N}(y_N)\\
&=&\prod_{i=1}^N f_i^0(y_i)+\sum_{\alpha_1,\cdots,\alpha_N\in \mathcal{I}} f_1^{\alpha_1}(y_1)\cdots f_N^{\alpha_N}(y_N),
\end{eqnarray*}
Let $c=\sum_{\alpha_1,\cdots,\alpha_N\in \mathcal{I}}T_m((b-b_{Q^*})f_1,f_2,\cdots,f_N)(x)$.
We have
\begin{eqnarray*}
II&\lesssim& \bigg(\frac{1}{|Q|}\int_Q |T_m((b-b_{Q^*})f_1^0,f_2^0,\cdots,f_N^0)(z)|^\delta dz\bigg)^{1/\delta}\\
&&\quad+\sum_{\alpha_1,\cdots, \alpha_N}\bigg(\frac{1}{|Q|}\int_Q |T_m((b-b_{Q^*})f_1^{\alpha_1},f_2^{\alpha_2},\cdots,f_N^{\alpha_N})(z)\\
&&\quad-T_m((b-b_{Q^*})f_1^{\alpha_1},f_2^{\alpha_2},\cdots,f_N^{\alpha_N})(x)|^\delta dz\bigg)^{1/\delta}\\
&=&II_1+\sum_{\alpha_1,\cdots,\alpha_N\in \mathcal{I}}II_{\alpha_1,\cdots,\alpha_N}.
\end{eqnarray*}
We first estimate $II_1$. By Kolmogorov's and H\"{o}lder's inequalities, we have
\begin{eqnarray*}
II_1&\lesssim& \|T_m((b-b_{Q^*})f_1^0,f_2^0,\cdots,f_N^0)\|_{L^{p_0/N,\infty}(Q,dx/{|Q|})}\\
&\lesssim& \bigg(\frac{1}{|Q^*|}\int_{Q^*} |(b-b_{Q^*})f_1(z)|^{p_0}dz\bigg)^{1/{p_0}}\prod_{i=2}^N
      \bigg(\frac{1}{|Q^*|}\int_{Q^*}|f_i(z)|^{p_0}dz\bigg)^{1/{p_0}}\\
&\lesssim& \|b\|^{}_{BMO}\mathcal{M}_{q_0}(\vec{f})(x).
\end{eqnarray*}
Next we estimate $II_{\alpha_1,\cdots,\alpha_N}$.
Analysis similar to that in the proof of Lemma~\ref{lm:main}
shows that 
\begin{eqnarray*}
&&\sum_{\alpha_1,\cdots,\alpha_N\in\mathcal{I}}II_{\alpha_1,\cdots,\alpha_N}\\
&\lesssim&\sum_{k=0}^\infty 2^{-k(s-Nn/{p_0})}\bigg(\frac{1}{|2^{k+1}Q^*|}\int_{2^{k+1}Q^*}|(b-b_{Q^*})f_1(y_1)|^{p_0}dy_1\bigg)^{1/{p_0}}\\
&&\quad\cdot \prod_{i=2}^N\bigg(\frac{1}{|2^{k+1}Q^*|}\int_{2^{k+1}Q^*}|f_i(y_i)|^{p_0}dy_i\bigg)^{1/{p_0}}\\
&\lesssim& \|b\|^{}_{BMO}\mathcal{M}_{q_0}(\vec{f})(x).
\end{eqnarray*}
This completes the proof.
\end{proof}

Now we are ready to prove Theorem~\ref{thm:main1}.
\begin{proof}[Proof of Theorem~\ref{thm:main1}]
By Proposition~\ref{prop:loptg}, there is some $1<r'<\min\{p_1/{p_0}$,
$\cdots$, $p_N/{p_0}\}$ such
that $\vec{w}\in A_{\vec{P}/{(p_0r')}}$. Let $q_0=p_0r'$. By Proposition~\ref{prop:bd}, we have
\[
  \|\mathcal{M}_{q_0}(\vec{f})\|_{L^p(v^{}_{\vec{w}})}\le C \prod_{i=1}^N \|f_i\|_{L^{p_i}(w_i)}.
\]
By Proposition~\ref{prop:p1} and Lemma~\ref{lm:main},
\begin{eqnarray*}
 \|M_\epsilon(T_m (\vec{f}))\|_{L^p(v_{\vec{w}})}
 &\le& C_{n,p,\delta, \vec{w}}\|M_\epsilon^\sharp(T_m (\vec{f}))\|_{L^p(v_{\vec{w}})}\\
 &\le& C\|\mathcal{M}_{p_0}(\vec{f})\|_{L^p(v_{\vec{w}})}.
\end{eqnarray*}
Then the desired conclusion follows from Proposition~\ref{prop:bd} and Lemma~\ref{lm:main1}.
\end{proof}

\begin{proof}[Proof of Corollary~\ref{cor:c2}]
By linearity it is enough to consider the case of $\vec{b}=(b$, $0$, $\cdots$, $0)\in BMO^N$. Fix $b\in BMO$ and consider the
operator
\[
  T_{m;b}(\vec{f})=bT_m(\vec{f})-T_m(bf_1,f_2,\cdots,f_N).
\]
Notice that
\begin{eqnarray*}
&&\int_{\bbR^n}T_{m;b}(f_1,\cdots,f_N)gdx \\
&=&
\int_{\bbR^n}T_{m_1}(bg,f_2,\cdots,f_N)f_1dx-\int_{\bbR^n}T_{m_1}(g,f_2,\cdots,f_N)b f_1dx\\
&=& -\int_{\bbR^n}T_{m_1;b}(g,\cdots,f_N)f_1dx.
\end{eqnarray*}
In much the same way as  in the proof of Corollary~\ref{cor:c1} we can get the conclusion desired.
\end{proof}

\end{document}